\documentclass{amsart}
\usepackage{amsmath}
\usepackage{amsfonts}

\newtheorem{theorem}{Theorem}
\theoremstyle{plain}

\newtheorem{corollary}{Corollary}

\newtheorem{lemma}{Lemma}

\newtheorem{proposition}{Proposition}

\numberwithin{equation}{section}
\input{tcilatex}

\begin{document}
\title[Integral inequalities]{On new general integral inequalities for
quasi-convex functions and their applications}
\author{\.{I}mdat \.{I}\c{s}can}
\address{Department of Mathematics, Faculty of Arts and Sciences,\\
Giresun University, 28100, Giresun, Turkey.}
\email{imdat.iscan@giresun.edu.tr}
\date{June 10, 2012}

\begin{abstract}
In this paper, we derive new estimates for the remainder term of the
midpoint, trapezoid, and Simpson formulae for functions whose derivatives in
absolute value at certain power are quasi-convex. Some applications to
special means of real numbers are also given.
\end{abstract}

\maketitle

\section{Introduction}

Let $f:I\subset \mathbb{R\rightarrow R}$ be a convex function defined on the
interval $I$ of real numbers and $a,b\in I$ with $a<b$. The following
inequality%
\begin{equation}
f\left( \frac{a+b}{2}\right) \leq \frac{1}{b-a}\dint\limits_{a}^{b}f(x)dx%
\leq \frac{f(a)+f(b)}{2}\text{.}  \label{1-1}
\end{equation}

holds. This double inequality is known in the literature as Hermite-Hadamard
integral inequality for convex functions. See \cite%
{AD10,ADK10,ADK11,DP00,I07,K04}, the results of the generalization,
improvement and extention of the famous integral inequality (\ref{1-1}).

The notion of quasi-convex functions generalizes the notion of convex
functions. More precisely, a function $f:[a,b]\mathbb{\rightarrow R}$ is
said quasi-convex on $[a,b]$ if 
\begin{equation*}
f\left( tx+(1-t)y\right) \leq \sup \left\{ f(x),f(y)\right\} ,
\end{equation*}%
for any $x,y\in \lbrack a,b]$ and $t\in \left[ 0,1\right] .$ Clearly, any
convex function is a quasi-convex function. Furthermore, there exist
quasi-convex functions which are not \ convex (see \cite{I07}).

The following inequality is well known in the literature as Simpson's
inequality .

Let $f:\left[ a,b\right] \mathbb{\rightarrow R}$ be a four times
continuously differentiable mapping on $\left( a,b\right) $ and $\left\Vert
f^{(4)}\right\Vert _{\infty }=\underset{x\in \left( a,b\right) }{\sup }%
\left\vert f^{(4)}(x)\right\vert <\infty .$ Then the following inequality
holds:%
\begin{equation*}
\left\vert \frac{1}{3}\left[ \frac{f(a)+f(b)}{2}+2f\left( \frac{a+b}{2}%
\right) \right] -\frac{1}{b-a}\dint\limits_{a}^{b}f(x)dx\right\vert \leq 
\frac{1}{2880}\left\Vert f^{(4)}\right\Vert _{\infty }\left( b-a\right) ^{2}.
\end{equation*}

\bigskip In recent years many authors have studied error estimations for
Simpson's inequality; for refinements, counterparts, generalizations and new
Simpson's type inequalities, see \cite{ADD09,AH11,SA11,SSO10,SSO10a}

In \cite{I07}, Ion introduced two inequalities of the right hand side of
Hadamard's type for quasi-convex functions, as follow:

\begin{theorem}
Assume $a,b\in \mathbb{R}$ with $a<b$ and $f:[a,b]\mathbb{\rightarrow R}$ is
a differentiable function on $\left( a,b\right) $. \ If $\left\vert
f^{\prime }\right\vert $ is quasi-convex on $[a,b],$ then the following
inequality holds true%
\begin{equation}
\left\vert \frac{f(a)+f(b)}{2}-\frac{1}{b-a}\dint\limits_{a}^{b}f(x)dx\right%
\vert \leq \frac{b-a}{4}\sup \left\{ \left\vert f^{\prime }(a)\right\vert
,\left\vert f^{\prime }(b)\right\vert \right\} .  \label{1-2}
\end{equation}
\end{theorem}

\begin{theorem}
Assume $a,b\in \mathbb{R}$ with $a<b$ and $f:[a,b]\mathbb{\rightarrow R}$ is
a differentiable function on $\left( a,b\right) $. Assume $p\in \mathbb{R}$
with $p>1.$ If $\left\vert f^{\prime }\right\vert ^{p/(p-1)}$ is
quasi-convex on $[a,b],$ then the following inequality holds true%
\begin{equation}
\left\vert \frac{f(a)+f(b)}{2}-\frac{1}{b-a}\dint\limits_{a}^{b}f(x)dx\right%
\vert \leq \frac{b-a}{2\left( p+1\right) ^{p/(p-1)}}\left( \sup \left\{
\left\vert f^{\prime }(a)\right\vert ^{\frac{p}{p-1}},\left\vert f^{\prime
}(b)\right\vert ^{\frac{p}{p-1}}\right\} \right) ^{\frac{p-1}{p}}.
\label{1-3}
\end{equation}
\end{theorem}

In \cite{ADK10}, Alomari et al. established some new upper bound for the
right -hand side of Hadamard's inequality for quasi-convex mappings, which
is the better than the inequality had done in \cite{I07}. The authors
obtained the following results:

\begin{theorem}
Let $f:I\subset \mathbb{R\rightarrow R}$ be a differentiable mapping on $%
I^{\circ }$ such that $f^{\prime }\in L[a,b]$, where $a,b\in I$ with $a<b.$
If $\left\vert f^{\prime }\right\vert ^{p/(p-1)}$ is an quasi-convex on $%
[a,b]$, for $p>1,$ then the following inequality holds: 
\begin{eqnarray}
\left\vert \frac{f(a)+f(b)}{2}-\frac{1}{b-a}\dint\limits_{a}^{b}f(x)dx\right%
\vert &\leq &\frac{b-a}{4\left( p+1\right) ^{1/p}}\left[ \left( \sup \left\{
\left\vert f^{\prime }(\frac{a+b}{2})\right\vert ^{\frac{p}{p-1}},\left\vert
f^{\prime }(b)\right\vert ^{\frac{p}{p-1}}\right\} \right) ^{\frac{p-1}{p}%
}\right.  \notag \\
&&\left. +\left( \sup \left\{ \left\vert f^{\prime }(\frac{a+b}{2}%
)\right\vert ^{\frac{p}{p-1}},\left\vert f^{\prime }(a)\right\vert ^{\frac{p%
}{p-1}}\right\} \right) ^{\frac{p-1}{p}}\right] .  \label{1-4}
\end{eqnarray}
\end{theorem}

\begin{theorem}
Let $f:I^{\circ }\subset \mathbb{R\rightarrow R}$ be a differentiable
mapping on $I^{\circ },$ $a,b\in I^{\circ }$ with $a<b.$ If $\left\vert
f^{\prime }\right\vert ^{q}$ is an quasi-convex on $[a,b]$, for $q\geq 1,$
then the following inequality holds:%
\begin{eqnarray}
\left\vert \frac{f(a)+f(b)}{2}-\frac{1}{b-a}\dint\limits_{a}^{b}f(x)dx\right%
\vert &\leq &\frac{b-a}{8}\left[ \left( \sup \left\{ \left\vert f^{\prime }(%
\frac{a+b}{2})\right\vert ^{q},\left\vert f^{\prime }(b)\right\vert
^{q}\right\} \right) ^{\frac{1}{q}}\right.  \notag \\
&&\left. +\left( \sup \left\{ \left\vert f^{\prime }(\frac{a+b}{2}%
)\right\vert ^{q},\left\vert f^{\prime }(a)\right\vert ^{q}\right\} \right)
^{\frac{1}{q}}\right] .  \label{1-5}
\end{eqnarray}
\end{theorem}

In this paper, in order to provide a unified approach to establish midpoint
inequality, trapezoid inequality and Simpson's inequality for functions
whose derivatives in absolute value at certain power are quasi-convex, we
need the following lemma given by Iscan in \cite{I12}:

\begin{lemma}
\label{1.1}Let $f:I\subset \mathbb{R\rightarrow R}$ be a differentiable
mapping on $I^{\circ }$ such that $f^{\prime }\in L[a,b]$, where $a,b\in I$
with $a<b$ and $\alpha ,\lambda \in \left[ 0,1\right] $. Then the following
equality holds:%
\begin{eqnarray}
&&\lambda \left( \alpha f(a)+\left( 1-\alpha \right) f(b)\right) +\left(
1-\lambda \right) f(\alpha a+\left( 1-\alpha \right) b)-\frac{1}{b-a}%
\dint\limits_{a}^{b}f(x)dx  \label{1-6} \\
&=&\left( b-a\right) \left[ \dint\limits_{0}^{1-\alpha }\left( t-\alpha
\lambda \right) f^{\prime }\left( tb+(1-t)a\right) dt\right.  \notag \\
&&\left. +\dint\limits_{1-\alpha }^{1}\left( t-1+\lambda \left( 1-\alpha
\right) \right) f^{\prime }\left( tb+(1-t)a\right) dt\right] .  \notag
\end{eqnarray}
\end{lemma}

\section{Main results}

\begin{theorem}
\label{2.1}Let $f:I\subset \mathbb{R\rightarrow R}$ be a differentiable
mapping on $I^{\circ }$ such that $f^{\prime }\in L[a,b]$, where $a,b\in
I^{\circ }$ with $a<b$ and $\alpha ,\lambda \in \left[ 0,1\right] $. If $%
\left\vert f^{\prime }\right\vert ^{q}$ is quasi-convex on $[a,b]$, $q\geq
1, $ then the following inequality holds:%
\begin{eqnarray}
&&\left\vert \lambda \left( \alpha f(a)+\left( 1-\alpha \right) f(b)\right)
+\left( 1-\lambda \right) f(\alpha a+\left( 1-\alpha \right) b)-\frac{1}{b-a}%
\dint\limits_{a}^{b}f(x)dx\right\vert  \notag \\
&\leq &\left\{ 
\begin{array}{cc}
\left( b-a\right) \left( \gamma _{2}+\upsilon _{2}\right) A^{\frac{1}{q}} & 
\alpha \lambda \leq 1-\alpha \leq 1-\lambda \left( 1-\alpha \right) \\ 
\left( b-a\right) \left( \gamma _{2}+\upsilon _{1}\right) A^{\frac{1}{q}} & 
\alpha \lambda \leq 1-\lambda \left( 1-\alpha \right) \leq 1-\alpha \\ 
\left( b-a\right) \left( \gamma _{1}+\upsilon _{2}\right) A^{\frac{1}{q}} & 
1-\alpha \leq \alpha \lambda \leq 1-\lambda \left( 1-\alpha \right)%
\end{array}%
\right.  \label{2-1}
\end{eqnarray}%
where 
\begin{equation*}
\gamma _{1}=\left( 1-\alpha \right) \left[ \alpha \lambda -\frac{\left(
1-\alpha \right) }{2}\right] ,\ \gamma _{2}=\left( \alpha \lambda \right)
^{2}-\gamma _{1}\ ,
\end{equation*}%
\begin{eqnarray*}
\upsilon _{1} &=&\frac{1-\left( 1-\alpha \right) ^{2}}{2}-\alpha \left[
1-\lambda \left( 1-\alpha \right) \right] , \\
\upsilon _{2} &=&\frac{1+\left( 1-\alpha \right) ^{2}}{2}-\left( \lambda
+1\right) \left( 1-\alpha \right) \left[ 1-\lambda \left( 1-\alpha \right) %
\right] ,
\end{eqnarray*}%
and%
\begin{equation*}
A=\sup \left\{ \left\vert f^{\prime }(a)\right\vert ^{q},\left\vert
f^{\prime }(b)\right\vert ^{q}\right\} .
\end{equation*}
\end{theorem}

\begin{proof}
Suppose that $q\geq 1.$ From Lemma \ref{1.1} and using the well known power
mean inequality, we have%
\begin{eqnarray*}
&&\left\vert \lambda \left( \alpha f(a)+\left( 1-\alpha \right) f(b)\right)
+\left( 1-\lambda \right) f(\alpha a+\left( 1-\alpha \right) b)-\frac{1}{b-a}%
\dint\limits_{a}^{b}f(x)dx\right\vert \\
&\leq &\left( b-a\right) \left[ \dint\limits_{0}^{1-\alpha }\left\vert
t-\alpha \lambda \right\vert \left\vert f^{\prime }\left( tb+(1-t)a\right)
\right\vert dt+\dint\limits_{1-\alpha }^{1}\left\vert t-1+\lambda \left(
1-\alpha \right) \right\vert \left\vert f^{\prime }\left( tb+(1-t)a\right)
\right\vert dt\right] \\
&\leq &\left( b-a\right) \left\{ \left( \dint\limits_{0}^{1-\alpha
}\left\vert t-\alpha \lambda \right\vert dt\right) ^{1-\frac{1}{q}}\left(
\dint\limits_{0}^{1-\alpha }\left\vert t-\alpha \lambda \right\vert
\left\vert f^{\prime }\left( tb+(1-t)a\right) \right\vert ^{q}dt\right) ^{%
\frac{1}{q}}\right.
\end{eqnarray*}%
\begin{equation}
\left. +\left( \dint\limits_{1-\alpha }^{1}\left\vert t-1+\lambda \left(
1-\alpha \right) \right\vert dt\right) ^{1-\frac{1}{q}}\left(
\dint\limits_{1-\alpha }^{1}\left\vert t-1+\lambda \left( 1-\alpha \right)
\right\vert \left\vert f^{\prime }\left( tb+(1-t)a\right) \right\vert
^{q}dt\right) ^{\frac{1}{q}}\right\} .  \label{2-3}
\end{equation}

Since $\left\vert f^{\prime }\right\vert ^{q}$ is quasi-convex on $[a,b],$
we know that for $t\in \left[ 0,1\right] $%
\begin{equation*}
\left\vert f^{\prime }\left( tb+(1-t)a\right) \right\vert ^{q}\leq \sup
\left\{ \left\vert f^{\prime }(a)\right\vert ^{q},\left\vert f^{\prime
}(b)\right\vert ^{q}\right\} ,
\end{equation*}%
hence, by simple computation%
\begin{equation}
\dint\limits_{0}^{1-\alpha }\left\vert t-\alpha \lambda \right\vert
dt=\left\{ 
\begin{array}{cc}
\gamma _{2}, & \alpha \lambda \leq 1-\alpha \\ 
\gamma _{1}, & \alpha \lambda \geq 1-\alpha%
\end{array}%
\right. ,  \label{2-4}
\end{equation}%
\begin{equation}
\dint\limits_{1-\alpha }^{1}\left\vert t-1+\lambda \left( 1-\alpha \right)
\right\vert dt=\left\{ 
\begin{array}{cc}
\upsilon _{1}, & 1-\lambda \left( 1-\alpha \right) \leq 1-\alpha \\ 
\upsilon _{2}, & 1-\lambda \left( 1-\alpha \right) \geq 1-\alpha%
\end{array}%
\right. ,  \label{2-5}
\end{equation}%
Thus, using (\ref{2-4}) and (\ref{2-5}) in (\ref{2-3}), we obtain the
inequality (\ref{2-1}). This completes the proof.
\end{proof}

\begin{corollary}
Under the assumptions of Theorem \ref{2.1} with $q=1,$ the inequality (\ref%
{2-1}) reduced to the following inequality%
\begin{equation*}
\left\vert \lambda \left( \alpha f(a)+\left( 1-\alpha \right) f(b)\right)
+\left( 1-\lambda \right) f(\alpha a+\left( 1-\alpha \right) b)-\frac{1}{b-a}%
\dint\limits_{a}^{b}f(x)dx\right\vert \leq
\end{equation*}%
\begin{equation*}
\left\{ 
\begin{array}{cc}
\left( b-a\right) \left( \gamma _{2}+\upsilon _{2}\right) \sup \left\{
\left\vert f^{\prime }(a)\right\vert ,\left\vert f^{\prime }(b)\right\vert
\right\} & \alpha \lambda \leq 1-\alpha \leq 1-\lambda \left( 1-\alpha
\right) \\ 
\left( b-a\right) \left( \gamma _{2}+\upsilon _{1}\right) \sup \left\{
\left\vert f^{\prime }(a)\right\vert ,\left\vert f^{\prime }(b)\right\vert
\right\} & \alpha \lambda \leq 1-\lambda \left( 1-\alpha \right) \leq
1-\alpha \\ 
\left( b-a\right) \left( \gamma _{1}+\upsilon _{2}\right) \sup \left\{
\left\vert f^{\prime }(a)\right\vert ,\left\vert f^{\prime }(b)\right\vert
\right\} & 1-\alpha \leq \alpha \lambda \leq 1-\lambda \left( 1-\alpha
\right)%
\end{array}%
\right. ,
\end{equation*}%
where where $\gamma _{1},\ \gamma _{2},\ \upsilon _{1}$ and$\ \upsilon _{2}$
are defined as in Theorem \ref{2.1}.
\end{corollary}

\begin{corollary}
Under the assumptions of Theorem \ref{2.1} with $\alpha =\frac{1}{2}$ and $%
\lambda =\frac{1}{3}$, from the inequality (\ref{2-1}) we get the following
Simpson type inequality 
\begin{eqnarray*}
&&\left\vert \frac{1}{6}\left[ f(a)+4f\left( \frac{a+b}{2}\right) +f(b)%
\right] -\frac{1}{b-a}\dint\limits_{a}^{b}f(x)dx\right\vert \\
&\leq &\left( b-a\right) \left( \frac{5}{36}\right) \sup \left\{ \left\vert
f^{\prime }(a)\right\vert ^{q},\left\vert f^{\prime }(b)\right\vert
^{q}\right\} .
\end{eqnarray*}
\end{corollary}

\begin{corollary}
Under the assumptions of Theorem \ref{2.1} with $\alpha =\frac{1}{2}$ and $%
\lambda =0,$from the inequality (\ref{2-1}) we get the following midpoint
inequality%
\begin{equation*}
\left\vert f\left( \frac{a+b}{2}\right) -\frac{1}{b-a}\dint%
\limits_{a}^{b}f(x)dx\right\vert \leq \frac{b-a}{4}\sup \left\{ \left\vert
f^{\prime }(a)\right\vert ^{q},\left\vert f^{\prime }(b)\right\vert
^{q}\right\} .
\end{equation*}
\end{corollary}

\begin{corollary}
Under the assumptions of Theorem \ref{2.1} with $\alpha =\frac{1}{2}$ and $%
\lambda =1,$from the inequality (\ref{2-1}) we get the following trapezoid
inequality%
\begin{equation*}
\left\vert \frac{f\left( a\right) +f\left( b\right) }{2}-\frac{1}{b-a}%
\dint\limits_{a}^{b}f(x)dx\right\vert \leq \frac{b-a}{4}\sup \left\{
\left\vert f^{\prime }(a)\right\vert ^{q},\left\vert f^{\prime
}(b)\right\vert ^{q}\right\} .
\end{equation*}%
which is the same of the inequality (\ref{1-2}) for $q=1$.
\end{corollary}

Using Lemma \ref{1.1} we shall give another result for convex functions as
follows.

\begin{theorem}
\label{2.2}Let $f:I\subset \mathbb{R\rightarrow R}$ be a differentiable
mapping on $I^{\circ }$ such that $f^{\prime }\in L[a,b]$, where $a,b\in
I^{\circ }$ with $a<b$ and $\alpha ,\lambda \in \left[ 0,1\right] $. If $%
\left\vert f^{\prime }\right\vert ^{q}$ is quasi-convex on $[a,b]$, $q>1,$
then the following inequality holds:%
\begin{equation}
\left\vert \lambda \left( \alpha f(a)+\left( 1-\alpha \right) f(b)\right)
+\left( 1-\lambda \right) f(\alpha a+\left( 1-\alpha \right) b)-\frac{1}{b-a}%
\dint\limits_{a}^{b}f(x)dx\right\vert \leq \left( b-a\right)  \label{2-6}
\end{equation}%
\begin{equation*}
\times \left( \frac{1}{p+1}\right) ^{\frac{1}{p}}A^{\frac{1}{q}}\left\{ 
\begin{array}{cc}
\left[ \left( 1-\alpha \right) ^{\frac{1}{q}}\varepsilon _{1}^{\frac{1}{p}%
}+\alpha ^{\frac{1}{q}}\varepsilon _{3}^{\frac{1}{p}}\right] , & \alpha
\lambda \leq 1-\alpha \leq 1-\lambda \left( 1-\alpha \right) \\ 
\left[ \left( 1-\alpha \right) ^{\frac{1}{q}}\varepsilon _{1}^{\frac{1}{p}%
}+\alpha ^{\frac{1}{q}}\varepsilon _{4}^{\frac{1}{p}}\right] , & \alpha
\lambda \leq 1-\lambda \left( 1-\alpha \right) \leq 1-\alpha \\ 
\left[ \left( 1-\alpha \right) ^{\frac{1}{q}}\varepsilon _{2}^{\frac{1}{p}%
}+\alpha ^{\frac{1}{q}}\varepsilon _{3}^{\frac{1}{p}}\right] , & 1-\alpha
\leq \alpha \lambda \leq 1-\lambda \left( 1-\alpha \right)%
\end{array}%
\right.
\end{equation*}%
where 
\begin{equation*}
A=\sup \left\{ \left\vert f^{\prime }(a)\right\vert ^{q},\left\vert
f^{\prime }(b)\right\vert ^{q}\right\} ,
\end{equation*}%
\begin{eqnarray}
\varepsilon _{1} &=&\left( \alpha \lambda \right) ^{p+1}+\left( 1-\alpha
-\alpha \lambda \right) ^{p+1},\ \varepsilon _{2}=\left( \alpha \lambda
\right) ^{p+1}-\left( \alpha \lambda -1+\alpha \right) ^{p+1},  \notag \\
\varepsilon _{3} &=&\left[ \lambda \left( 1-\alpha \right) \right] ^{p+1}+%
\left[ \alpha -\lambda \left( 1-\alpha \right) \right] ^{p+1},\ \varepsilon
_{4}=\left[ \lambda \left( 1-\alpha \right) \right] ^{p+1}-\left[ \lambda
\left( 1-\alpha \right) -\alpha \right] ^{p+1},  \notag
\end{eqnarray}%
and $\frac{1}{p}+\frac{1}{q}=1.$
\end{theorem}

\begin{proof}
From Lemma \ref{2.1} and by H\"{o}lder's integral inequality, we have%
\begin{eqnarray*}
&&\left\vert \lambda \left( \alpha f(a)+\left( 1-\alpha \right) f(b)\right)
+\left( 1-\lambda \right) f(\alpha a+\left( 1-\alpha \right) b)-\frac{1}{b-a}%
\dint\limits_{a}^{b}f(x)dx\right\vert \\
&\leq &\left( b-a\right) \left[ \dint\limits_{0}^{1-\alpha }\left\vert
t-\alpha \lambda \right\vert \left\vert f^{\prime }\left( tb+(1-t)a\right)
\right\vert dt+\dint\limits_{1-\alpha }^{1}\left\vert t-1+\lambda \left(
1-\alpha \right) \right\vert \left\vert f^{\prime }\left( tb+(1-t)a\right)
\right\vert dt\right] \\
&\leq &\left( b-a\right) \left\{ \left( \dint\limits_{0}^{1-\alpha
}\left\vert t-\alpha \lambda \right\vert ^{p}dt\right) ^{\frac{1}{p}}\left(
\dint\limits_{0}^{1-\alpha }\left\vert f^{\prime }\left( tb+(1-t)a\right)
\right\vert ^{q}dt\right) ^{\frac{1}{q}}\right.
\end{eqnarray*}%
\begin{equation}
+\left. \left( \dint\limits_{1-\alpha }^{1}\left\vert t-1+\lambda \left(
1-\alpha \right) \right\vert ^{p}dt\right) ^{\frac{1}{p}}\left(
\dint\limits_{1-\alpha }^{1}\left\vert f^{\prime }\left( tb+(1-t)a\right)
\right\vert ^{q}dt\right) ^{\frac{1}{q}}\right\} .  \label{2-7}
\end{equation}%
Since $\left\vert f^{\prime }\right\vert ^{q}$ is quasi-convex on $[a,b],$
for $\alpha \in \left[ 0,1\right] $, we get 
\begin{equation}
\dint\limits_{0}^{1-\alpha }\left\vert f^{\prime }\left( tb+(1-t)a\right)
\right\vert ^{q}dt=\left( 1-\alpha \right) \sup \left\{ \left\vert f^{\prime
}(a)\right\vert ^{q},\left\vert f^{\prime }(b)\right\vert ^{q}\right\}
\label{2-8}
\end{equation}
Similarly, for $\alpha \in \left[ 0,1\right] $, we have 
\begin{equation}
\dint\limits_{1-\alpha }^{1}\left\vert f^{\prime }\left( tb+(1-t)a\right)
\right\vert ^{q}dt=\alpha \sup \left\{ \left\vert f^{\prime }(a)\right\vert
^{q},\left\vert f^{\prime }(b)\right\vert ^{q}\right\} .  \label{2-9}
\end{equation}
By simple computation%
\begin{equation}
\dint\limits_{0}^{1-\alpha }\left\vert t-\alpha \lambda \right\vert
^{p}dt=\left\{ 
\begin{array}{cc}
\frac{\left( \alpha \lambda \right) ^{p+1}+\left( 1-\alpha -\alpha \lambda
\right) ^{p+1}}{p+1}, & \alpha \lambda \leq 1-\alpha \\ 
\frac{\left( \alpha \lambda \right) ^{p+1}-\left( \alpha \lambda -1+\alpha
\right) ^{p+1}}{p+1}, & \alpha \lambda \geq 1-\alpha%
\end{array}%
\right. ,  \label{2-10}
\end{equation}%
and%
\begin{equation}
\dint\limits_{1-\alpha }^{1}\left\vert t-1+\lambda \left( 1-\alpha \right)
\right\vert ^{p}dt=\left\{ 
\begin{array}{cc}
\frac{\left[ \lambda \left( 1-\alpha \right) \right] ^{p+1}+\left[ \alpha
-\lambda \left( 1-\alpha \right) \right] ^{p+1}}{p+1}, & 1-\alpha \leq
1-\lambda \left( 1-\alpha \right) \\ 
\frac{\left[ \lambda \left( 1-\alpha \right) \right] ^{p+1}-\left[ \lambda
\left( 1-\alpha \right) -\alpha \right] ^{p+1}}{p+1}, & 1-\alpha \geq
1-\lambda \left( 1-\alpha \right)%
\end{array}%
\right. ,  \label{2-11}
\end{equation}%
thus, using (\ref{2-8})-(\ref{2-11}) in (\ref{2-7}), we obtain the
inequality (\ref{2-6}). This completes the proof.
\end{proof}

\begin{corollary}
Under the assumptions of Theorem \ref{2.2} with $\alpha =\frac{1}{2}$ and $%
\lambda =\frac{1}{3}$, from the inequality (\ref{2-6}) we get the following
Simpson type inequality 
\begin{equation*}
\left\vert \frac{1}{6}\left[ f(a)+4f\left( \frac{a+b}{2}\right) +f(b)\right]
-\frac{1}{b-a}\dint\limits_{a}^{b}f(x)dx\right\vert
\end{equation*}%
\begin{equation*}
\leq \frac{b-a}{6}\left( \frac{1+2^{p+1}}{3\left( p+1\right) }\right) ^{%
\frac{1}{p}}\left( \sup \left\{ \left\vert f^{\prime }(a)\right\vert
^{q},\left\vert f^{\prime }(b)\right\vert ^{q}\right\} \right) ^{\frac{1}{q}%
}.
\end{equation*}
\end{corollary}

\begin{corollary}
Under the assumptions of Theorem \ref{2.2} with $\alpha =\frac{1}{2}$ and $%
\lambda =0,$ from the inequality (\ref{2-6}) we get the following midpoint
inequality%
\begin{equation*}
\left\vert f\left( \frac{a+b}{2}\right) -\frac{1}{b-a}\dint%
\limits_{a}^{b}f(x)dx\right\vert \leq \frac{b-a}{4}\left( \frac{1}{p+1}%
\right) ^{\frac{1}{p}}\left( \sup \left\{ \left\vert f^{\prime
}(a)\right\vert ^{q},\left\vert f^{\prime }(b)\right\vert ^{q}\right\}
\right) ^{\frac{1}{q}}.
\end{equation*}
\end{corollary}

\begin{corollary}
Let the assumptions of Theorem \ref{2.2} hold. Then for $\alpha =\frac{1}{2}$
and $\lambda =1,$ from the inequality (\ref{2-6}) we get the following
trapezoid inequality%
\begin{equation*}
\left\vert \frac{f\left( a\right) +f\left( b\right) }{2}-\frac{1}{b-a}%
\dint\limits_{a}^{b}f(x)dx\right\vert \leq \frac{b-a}{4}\left( \frac{1}{p+1}%
\right) ^{\frac{1}{p}}\left( \sup \left\{ \left\vert f^{\prime
}(a)\right\vert ^{q},\left\vert f^{\prime }(b)\right\vert ^{q}\right\}
\right) ^{\frac{1}{q}}.
\end{equation*}
\end{corollary}

\begin{theorem}
\label{2.3}Let $f:I\subset \mathbb{R\rightarrow R}$ be a differentiable
mapping on $I^{\circ }$ such that $f^{\prime }\in L[a,b]$, where $a,b\in
I^{\circ }$ with $a<b$ and $\alpha ,\lambda \in \left[ 0,1\right] $. If $%
\left\vert f^{\prime }\right\vert ^{q}$ is quasi-convex on $[a,b]$, $q>1,$
then the following inequality holds:%
\begin{equation}
\left\vert \lambda \left( \alpha f(a)+\left( 1-\alpha \right) f(b)\right)
+\left( 1-\lambda \right) f(\alpha a+\left( 1-\alpha \right) b)-\frac{1}{b-a}%
\dint\limits_{a}^{b}f(x)dx\right\vert \leq \left( b-a\right)   \label{2-12}
\end{equation}%
\begin{equation*}
\times \left( \frac{1}{p+1}\right) ^{\frac{1}{p}}\left\{ 
\begin{array}{cc}
\left[ \left( 1-\alpha \right) ^{\frac{1}{q}}B^{\frac{1}{q}}\varepsilon
_{1}^{\frac{1}{p}}+\alpha ^{\frac{1}{q}}C^{\frac{1}{q}}\varepsilon _{3}^{%
\frac{1}{p}}\right] , & \alpha \lambda \leq 1-\alpha \leq 1-\lambda \left(
1-\alpha \right)  \\ 
\left[ \left( 1-\alpha \right) ^{\frac{1}{q}B^{\frac{1}{q}}}\varepsilon
_{1}^{\frac{1}{p}}+\alpha ^{\frac{1}{q}}C^{\frac{1}{q}}\varepsilon _{4}^{%
\frac{1}{p}}\right] , & \alpha \lambda \leq 1-\lambda \left( 1-\alpha
\right) \leq 1-\alpha  \\ 
\left[ \left( 1-\alpha \right) ^{\frac{1}{q}}B^{\frac{1}{q}}\varepsilon
_{2}^{\frac{1}{p}}+\alpha ^{\frac{1}{q}}C^{\frac{1}{q}}\varepsilon _{3}^{%
\frac{1}{p}}\right] , & 1-\alpha \leq \alpha \lambda \leq 1-\lambda \left(
1-\alpha \right) 
\end{array}%
\right. 
\end{equation*}%
where 
\begin{eqnarray*}
B &=&\sup \left\{ \left\vert f^{\prime }(a)\right\vert ^{q},\left\vert
f^{\prime }(\alpha a+\left( 1-\alpha \right) b)\right\vert ^{q}\right\} , \\
C &=&\sup \left\{ \left\vert f^{\prime }(b)\right\vert ^{q},\left\vert
f^{\prime }(\alpha a+\left( 1-\alpha \right) b)\right\vert ^{q}\right\} ,
\end{eqnarray*}%
, $\frac{1}{p}+\frac{1}{q}=1$, and $\varepsilon _{1},\ \varepsilon _{2},\
\varepsilon _{3},\ \varepsilon _{4}$ numbers are defined as in Theorem \ref%
{2.2}.
\end{theorem}

\begin{proof}
From Lemma \ref{1.1} and by H\"{o}lder's integral inequality, we have the
inequality (\ref{2-7}). Since $\left\vert f^{\prime }\right\vert ^{q}$ is
convex on $[a,b],$ for all $t\in \left[ 0,1\right] $ and $\alpha \in \left[
0,1\right) $ we get%
\begin{equation*}
\left\vert f^{\prime }\left( ta+(1-t)\left[ \alpha a+\left( 1-\alpha \right)
b\right] \right) \right\vert ^{q}\leq B=\sup \left\{ \left\vert f^{\prime
}(a)\right\vert ^{q},\left\vert f^{\prime }(\alpha a+\left( 1-\alpha \right)
b)\right\vert ^{q}\right\}
\end{equation*}
then%
\begin{equation}
\dint\limits_{0}^{1}\left\vert f^{\prime }\left( ta+(1-t)\left[ \alpha
a+\left( 1-\alpha \right) b\right] \right) \right\vert ^{q}dt=\frac{1}{%
\left( 1-\alpha \right) \left( b-a\right) }\dint\limits_{a}^{\left( 1-\alpha
\right) b+\alpha a}\left\vert f^{\prime }\left( x\right) \right\vert
^{q}dx\leq B.  \label{2-13}
\end{equation}
By the inequality (\ref{2-13}), we get 
\begin{eqnarray}
\dint\limits_{0}^{1-\alpha }\left\vert f^{\prime }\left( tb+(1-t)a\right)
\right\vert ^{q}dt &=&\left( 1-\alpha \right) \left[ \frac{1}{\left(
1-\alpha \right) \left( b-a\right) }\dint\limits_{a}^{\left( 1-\alpha
\right) b+\alpha a}\left\vert f^{\prime }\left( x\right) \right\vert ^{q}dx%
\right]  \notag \\
&\leq &\left( 1-\alpha \right) B.  \label{2-14}
\end{eqnarray}%
The inequality (\ref{2-14}) also holds for $\alpha =1$ too. Since $%
\left\vert f^{\prime }\right\vert ^{q}$ is convex on $[a,b],$ for all $t\in %
\left[ 0,1\right] $ and $\alpha \in \left( 0,1\right] $ we have%
\begin{equation*}
\left\vert f^{\prime }\left( tb+(1-t)\left[ \alpha a+\left( 1-\alpha \right)
b\right] \right) \right\vert ^{q}\leq C=\sup \left\{ \left\vert f^{\prime
}(b)\right\vert ^{q},\left\vert f^{\prime }(\alpha a+\left( 1-\alpha \right)
b)\right\vert ^{q}\right\}
\end{equation*}%
then%
\begin{equation}
\dint\limits_{0}^{1}\left\vert f^{\prime }\left( tb+(1-t)\left[ \alpha
a+\left( 1-\alpha \right) b\right] \right) \right\vert ^{q}dt=\frac{1}{%
\alpha \left( b-a\right) }\dint\limits_{\left( 1-\alpha \right) b+\alpha
a}^{b}\left\vert f^{\prime }\left( x\right) \right\vert ^{q}dx\leq B.
\label{2-15}
\end{equation}%
By the inequality (\ref{2-15}), we get 
\begin{eqnarray}
\dint\limits_{1-\alpha }^{1}\left\vert f^{\prime }\left( tb+(1-t)a\right)
\right\vert ^{q}dt &=&\alpha \left[ \frac{1}{\alpha \left( b-a\right) }%
\dint\limits_{\left( 1-\alpha \right) b+\alpha a}^{b}\left\vert f^{\prime
}\left( x\right) \right\vert ^{q}dx\right]  \notag \\
&\leq &\alpha C.  \label{2-16}
\end{eqnarray}%
The inequality (\ref{2-16}) also holds for $\alpha =0$ too. Thus, using (\ref%
{2-10}), (\ref{2-11}), (\ref{2-14}) and (\ref{2-16}) in (\ref{2-7}), we
obtain the inequality (\ref{2-12}). This completes the proof.
\end{proof}

\begin{corollary}
Under the assumptions of Theorem \ref{2.3} with $\alpha =\frac{1}{2}$ and $%
\lambda =\frac{1}{3}$, from the inequality (\ref{2-12}) we get the following
Simpson type inequality 
\begin{equation*}
\left\vert \frac{1}{6}\left[ f(a)+4f\left( \frac{a+b}{2}\right) +f(b)\right]
-\frac{1}{b-a}\dint\limits_{a}^{b}f(x)dx\right\vert
\end{equation*}%
\begin{eqnarray*}
&\leq &\frac{b-a}{12}\left( \frac{1+2^{p+1}}{3\left( p+1\right) }\right) ^{%
\frac{1}{p}}\left\{ \left( \sup \left\{ \left\vert f^{\prime }(\frac{a+b}{2}%
)\right\vert ^{q},\left\vert f^{\prime }(a)\right\vert ^{q}\right\} \right)
^{\frac{1}{q}}\right. \\
&&\left. +\left( \sup \left\{ \left\vert f^{\prime }(\frac{a+b}{2}%
)\right\vert ^{q},\left\vert f^{\prime }(b)\right\vert ^{q}\right\} \right)
^{\frac{1}{q}}\right\} .
\end{eqnarray*}
\end{corollary}

\begin{corollary}
Under the assumptions of Theorem \ref{2.3} with $\alpha =\frac{1}{2}$ and $%
\lambda =1$, from the inequality (\ref{2-12}) we get the following trapezoid
inequality%
\begin{eqnarray*}
\left\vert \frac{f(a)+f(b)}{2}-\frac{1}{b-a}\dint\limits_{a}^{b}f(x)dx\right%
\vert &\leq &\frac{b-a}{4\left( p+1\right) ^{1/p}}\left[ \left\{ \left( \sup
\left\{ \left\vert f^{\prime }(\frac{a+b}{2})\right\vert ^{q},\left\vert
f^{\prime }(a)\right\vert ^{q}\right\} \right) ^{\frac{1}{q}}\right. \right.
\\
&&\left. +\left\{ \left( \sup \left\{ \left\vert f^{\prime }(\frac{a+b}{2}%
)\right\vert ^{q},\left\vert f^{\prime }(b)\right\vert ^{q}\right\} \right)
^{\frac{1}{q}}\right. \right] .
\end{eqnarray*}%
which is the same of the inequality (\ref{1-4}).
\end{corollary}

\begin{corollary}
Under the assumptions of Theorem \ref{2.3} with $\alpha =\frac{1}{2}$ and $%
\lambda =0$, from the inequality (\ref{2-12}) we get the following midpoint
inequality%
\begin{eqnarray*}
\left\vert f\left( \frac{a+b}{2}\right) -\frac{1}{b-a}\dint%
\limits_{a}^{b}f(x)dx\right\vert &\leq &\frac{b-a}{4\left( p+1\right) ^{1/p}}%
\left[ \left\{ \left( \sup \left\{ \left\vert f^{\prime }(\frac{a+b}{2}%
)\right\vert ^{q},\left\vert f^{\prime }(a)\right\vert ^{q}\right\} \right)
^{\frac{1}{q}}\right. \right. \\
&&\left. +\left\{ \left( \sup \left\{ \left\vert f^{\prime }(\frac{a+b}{2}%
)\right\vert ^{q},\left\vert f^{\prime }(b)\right\vert ^{q}\right\} \right)
^{\frac{1}{q}}\right. \right] ,
\end{eqnarray*}%
which is the better than the inequality in \cite[Corollary 8]{AD10}.
\end{corollary}

\section{Some applications for special means}

Let us recall the following special means of arbitrary real numbers $a,b$
with $a\neq b$ and $\alpha \in \left[ 0,1\right] :$

\begin{enumerate}
\item The weighted arithmetic mean%
\begin{equation*}
A_{\alpha }\left( a,b\right) :=\alpha a+(1-\alpha )b,~a,b\in 
\mathbb{R}
.
\end{equation*}

\item The unweighted arithmetic mean%
\begin{equation*}
A\left( a,b\right) :=\frac{a+b}{2},~a,b\in 
\mathbb{R}
.
\end{equation*}

\item The weighted harmonic mean%
\begin{equation*}
H_{\alpha }\left( a,b\right) :=\left( \frac{\alpha }{a}+\frac{1-\alpha }{b}%
\right) ^{-1},\ \ a,b\in 
\mathbb{R}
\backslash \left\{ 0\right\} .
\end{equation*}

\item The unweighted harmonic mean%
\begin{equation*}
H\left( a,b\right) :=\frac{2ab}{a+b},\ \ a,b\in 
\mathbb{R}
\backslash \left\{ 0\right\} .
\end{equation*}

\item The Logarithmic mean%
\begin{equation*}
L\left( a,b\right) :=\frac{b-a}{\ln b-\ln a},\ \ a,b>0,\ a\neq b\ .
\end{equation*}

\item Then n-Logarithmic mean%
\begin{equation*}
L_{n}\left( a,b\right) :=\ \left( \frac{b^{n+1}-a^{n+1}}{(n+1)(b-a)}\right)
^{\frac{1}{n}}\ ,\ n\in 
\mathbb{N}
,\ a,b\in 
\mathbb{R}
,\ a\neq b.
\end{equation*}
\end{enumerate}

\begin{proposition}
Let $a,b\in 
\mathbb{R}
$ with $a<b,$ and $n\in 
\mathbb{N}
,\ n\geq 2.$ Then, for $\alpha ,\lambda \in \left[ 0,1\right] $ and $q\geq
1, $we have the following inequality:%
\begin{eqnarray*}
&&\left\vert \lambda A_{\alpha }\left( a^{n},b^{n}\right) +\left( 1-\lambda
\right) A_{\alpha }^{n}\left( a,b\right) -L_{n}^{n}\left( a,b\right)
\right\vert \\
&\leq &\left\{ 
\begin{array}{cc}
n\left( b-a\right) \left( \gamma _{2}+\upsilon _{2}\right) E^{\frac{1}{q}} & 
\alpha \lambda \leq 1-\alpha \leq 1-\lambda \left( 1-\alpha \right) \\ 
n\left( b-a\right) \left( \gamma _{2}+\upsilon _{1}\right) E^{\frac{1}{q}} & 
\alpha \lambda \leq 1-\lambda \left( 1-\alpha \right) \leq 1-\alpha \\ 
n\left( b-a\right) \left( \gamma _{1}+\upsilon _{2}\right) E^{\frac{1}{q}} & 
1-\alpha \leq \alpha \lambda \leq 1-\lambda \left( 1-\alpha \right)%
\end{array}%
\right. ,
\end{eqnarray*}%
where%
\begin{equation*}
E=\sup \left\{ \left\vert a\right\vert ^{(n-1)q},\left\vert b\right\vert
^{(n-1)q}\right\} ,
\end{equation*}%
$\gamma _{1},\ \gamma _{2},\ \upsilon _{1}\ $and $\ \upsilon _{2}$ are
defined as in Theorem \ref{2.1}.
\end{proposition}

\begin{proof}
The assertion follows from Theorem \ref{2.1}, for$\ f(x)=x^{n},\ x\in 
\mathbb{R}
.$
\end{proof}

\begin{proposition}
Let $a,b\in 
\mathbb{R}
$ with $a<b,$ and $n\in 
\mathbb{N}
,\ n\geq 2.$ Then, for $\alpha ,\lambda \in \left[ 0,1\right] $ and $q>1,$we
have the following inequality:%
\begin{equation*}
\left\vert \lambda A_{\alpha }\left( a^{n},b^{n}\right) +\left( 1-\lambda
\right) A_{\alpha }^{n}\left( a,b\right) -L_{n}^{n}\left( a,b\right)
\right\vert \leq \left( b-a\right) \left( \frac{1}{p+1}\right) ^{\frac{1}{p}%
}n
\end{equation*}%
\begin{equation*}
\times \left\{ 
\begin{array}{cc}
\left[ \left( 1-\alpha \right) ^{\frac{1}{q}}F^{\frac{1}{q}}\varepsilon
_{1}^{\frac{1}{p}}+\alpha ^{\frac{1}{q}}G^{\frac{1}{q}}\varepsilon _{3}^{%
\frac{1}{p}}\right] , & \alpha \lambda \leq 1-\alpha \leq 1-\lambda \left(
1-\alpha \right)  \\ 
\left[ \left( 1-\alpha \right) ^{\frac{1}{q}}F^{\frac{1}{q}}\varepsilon
_{1}^{\frac{1}{p}}+\alpha ^{\frac{1}{q}}G^{\frac{1}{q}}\varepsilon _{4}^{%
\frac{1}{p}}\right] , & \alpha \lambda \leq 1-\lambda \left( 1-\alpha
\right) \leq 1-\alpha  \\ 
\left[ \left( 1-\alpha \right) ^{\frac{1}{q}}F^{\frac{1}{q}}\varepsilon
_{2}^{\frac{1}{p}}+\alpha ^{\frac{1}{q}}G^{\frac{1}{q}}\varepsilon _{3}^{%
\frac{1}{p}}\right] , & 1-\alpha \leq \alpha \lambda \leq 1-\lambda \left(
1-\alpha \right) 
\end{array}%
\right. ,
\end{equation*}%
where 
\begin{eqnarray*}
F &=&\sup \left\{ \left\vert a\right\vert ^{(n-1)q},\left\vert A_{\alpha
}\left( a,b\right) \right\vert ^{(n-1)q}\right\} ,\  \\
G &=&\sup \left\{ \left\vert b\right\vert ^{(n-1)q},\left\vert A_{\alpha
}\left( a,b\right) \right\vert ^{(n-1)q}\right\} ,\ 
\end{eqnarray*}%
$\frac{1}{p}+\frac{1}{q}=1,$ and $\varepsilon _{1},\ \varepsilon _{2},\
\varepsilon _{3},\ \varepsilon _{4}$ numbers are defined as in Theorem \ref%
{2.2}.
\end{proposition}

\begin{proof}
The assertion follows from Theorem \ref{2.3}, for$\ f(x)=x^{n},\ x\in 
\mathbb{R}
.$
\end{proof}

\begin{proposition}
Let $a,b\in 
\mathbb{R}
$ with $a<b,\ 0\notin \left[ a,b\right] .$ Then, for $\alpha ,\lambda \in %
\left[ 0,1\right] $ and $q\geq 1,$ we have the following inequality:%
\begin{eqnarray*}
&&\left\vert \lambda H_{\alpha }^{-1}\left( a,b\right) +\left( 1-\lambda
\right) A_{\alpha }^{-1}\left( a,b\right) -L^{-1}\left( a,b\right)
\right\vert \\
&\leq &\left\{ 
\begin{array}{cc}
\left( b-a\right) \left( \gamma _{2}+\upsilon _{2}\right) K^{\frac{1}{q}} & 
\alpha \lambda \leq 1-\alpha \leq 1-\lambda \left( 1-\alpha \right) \\ 
\left( b-a\right) \left( \gamma _{2}+\upsilon _{1}\right) K^{\frac{1}{q}} & 
\alpha \lambda \leq 1-\lambda \left( 1-\alpha \right) \leq 1-\alpha \\ 
\left( b-a\right) \left( \gamma _{1}+\upsilon _{2}\right) K^{\frac{1}{q}} & 
1-\alpha \leq \alpha \lambda \leq 1-\lambda \left( 1-\alpha \right)%
\end{array}%
\right. ,
\end{eqnarray*}%
where%
\begin{equation*}
K=\sup \left\{ a^{-2q},b^{-2q}\right\} ,
\end{equation*}%
$\gamma _{1},\ \gamma _{2},\ \upsilon _{1},\ $and $\ \upsilon _{2}$ are
defined as in Theorem \ref{2.1}..
\end{proposition}

\begin{proof}
The assertion follows from Theorem \ref{2.1}., for$\ f(x)=\frac{1}{x},\ x\in
\left( 0,\infty \right) .$
\end{proof}

\begin{proposition}
Let $a,b\in 
\mathbb{R}
$ with $0<a<b.$ Then, for $\alpha ,\lambda \in \left[ 0,1\right] $ and $q>1,$
we have the following inequality:%
\begin{equation*}
\left\vert \lambda H_{\alpha }^{-1}\left( a,b\right) +\left( 1-\lambda
\right) A_{\alpha }^{-1}\left( a,b\right) -L^{-1}\left( a,b\right)
\right\vert \leq \left( b-a\right) \left( \frac{1}{p+1}\right) ^{\frac{1}{p}}
\end{equation*}%
\begin{equation*}
\times \left\{ 
\begin{array}{cc}
\left[ \left( 1-\alpha \right) ^{\frac{1}{q}}M^{\frac{1}{q}}\varepsilon
_{1}^{\frac{1}{p}}+\alpha ^{\frac{1}{q}}N^{\frac{1}{q}}\varepsilon _{3}^{%
\frac{1}{p}}\right] , & \alpha \lambda \leq 1-\alpha \leq 1-\lambda \left(
1-\alpha \right) \\ 
\left[ \left( 1-\alpha \right) ^{\frac{1}{q}}M^{\frac{1}{q}}\varepsilon
_{1}^{\frac{1}{p}}+\alpha ^{\frac{1}{q}}N^{\frac{1}{q}}\varepsilon _{4}^{%
\frac{1}{p}}\right] , & \alpha \lambda \leq 1-\lambda \left( 1-\alpha
\right) \leq 1-\alpha \\ 
\left[ \left( 1-\alpha \right) ^{\frac{1}{q}}M^{\frac{1}{q}}\varepsilon
_{2}^{\frac{1}{p}}+\alpha ^{\frac{1}{q}}N^{\frac{1}{q}}\varepsilon _{3}^{%
\frac{1}{p}}\right] , & 1-\alpha \leq \alpha \lambda \leq 1-\lambda \left(
1-\alpha \right)%
\end{array}%
\right. ,
\end{equation*}%
where 
\begin{eqnarray*}
M &=&\sup \left\{ a^{-2q},A_{\alpha }\left( a,b\right) ^{-2q}\right\} ,\  \\
N &=&\sup \left\{ b^{-2q},A_{\alpha }\left( a,b\right) ^{-2q}\right\} ,\ 
\end{eqnarray*}%
$\frac{1}{p}+\frac{1}{q}=1,$ and $\varepsilon _{1},\ \varepsilon _{2},\
\varepsilon _{3},\ $and $\varepsilon _{4}$ are defined as in Theorem \ref%
{2.3}.
\end{proposition}

\begin{proof}
The assertion follows from Theorem \ref{2.3}, for$\ f(x)=\frac{1}{x}$,$\
x\in \left( 0,\infty \right) .$
\end{proof}

\end{document}